\documentclass[12pt,centertags,oneside]{amsart}

\usepackage{amscd,amsxtra,calc}
\usepackage{cmmib57}
\usepackage{url}

\usepackage{amscd}
\usepackage{pstricks}
\usepackage{color}
\usepackage[a4paper,width=16.2cm,top=3cm,bottom=3cm]{geometry}

\numberwithin{equation}{section}

\theoremstyle{plain}
\newtheorem{thm}{Theorem}[section]
\newtheorem{theorem}[thm]{Theorem}

\newtheorem{lemma}[thm]{Lemma}
\newtheorem{corollary}[thm]{Corollary}
\newtheorem{proposition}[thm]{Proposition}
\theoremstyle{definition}
\newtheorem{remark}[thm]{Remark}

\newtheorem{assumption}[thm]{Assumption}

\numberwithin{equation}{section}
\newcommand{\C}{{\mathbb C}}

\renewcommand{\P}{{\mathbb P}}
\newcommand{\Q}{{\mathbb Q}}

\newcommand{\Z}{{\mathbb Z}}

\newcommand{\id}{{\rm id\hspace{.1ex}}}

\title [non-finitely generated automorphism group]{A surface birational to an Enriques surface with non-finitely generated automorphism group}

\author{JongHae Keum}
\address{Korea Institute for Advanced Study, Hoegiro 85, Seoul,
02455, Korea}
\email{jhkeum@kias.re.kr}

\author{Keiji Oguiso}
\address{Mathematical Sciences, the University of Tokyo, Meguro Komaba 3-8-1, Tokyo, Japan; Korea Institute for Advanced Study, Hoegiro 85, Seoul, 02455, Korea; National Center for Theoretical Sciences
No. 1 Sec. 4 Roosevelt Rd., Taipei, 106, Taiwan
}
\email{oguiso@ms.u-tokyo.ac.jp}

\thanks{Keum was supported by the National Research Foundation of Korea (NRF 2019R1A2C3010487).
Oguiso was supported by JSPS Grant-in-Aid (S) 15H05738, JSPS Grant-in-Aid (B) 15H03611, and by KIAS Scholar Program. }

\dedicatory{Dedicated to Professor Shigeru Mukai on the occasion of his 65th birthday}

\begin{document}

\maketitle
\begin{abstract}
We will show that there is a smooth complex projective surface, birational to
some Enriques surface, such that the automorphism group is %necessarily
discrete but not finitely generated.
\end{abstract}

\section{Introduction}

We work over the complex number field $\C$. A K3 surface is a compact simply connected, in the classical topology, smooth complex surface with nowhere vanishing global holomorphic $2$-form. An Enriques surface is a smooth complex surface which is isomorphic to a non-trivial \'etale quotient of a K3 surface. The quotient map is necessarily of degree two and every Enriques surface is projective.

Our main theorem is the following:

\begin{theorem}\label{thm1}

There is a smooth projective surface $Y$ birational to some Enriques surface such that ${\rm Aut}\, (Y)$ is not finitely generated.

\end{theorem}

\begin{remark}\label{rem1} Let $Y$ be a smooth projective surface birational to an Enriques surface $S$ and let $\tilde{S}$ be the universal covering K3 surface of $S$.
\begin{enumerate}
\item ${\rm Aut}^0(S) = \{\id_S\}$, i.e., ${\rm Aut}\,(S)$ is discrete. This is because $H^0(S, T_{S}) = 0$ by $H^0(\tilde{S}, T_{\tilde{S}}) = 0$. On the other hand, ${\rm Aut}\, (S)$ itself is finitely generated. This is because, "up to finite kernel and cokernel", ${\rm Aut}\, (S)$ is isomorphic to the quotient group ${\rm O}({\rm NS}\, (S)/{\rm torsion})/W(S)$ of the arithmetic subgroup ${\rm O}({\rm NS}\, (S)/{\rm torsion})$ by the Weyl group $W(S)$ generated by the reflections corresponding to the smooth rational curves on $S$ (see \cite[Theorem]{Do84} for a more precise statement) and ${\rm O}({\rm NS}\, (S)/{\rm torsion})$ is finitely generated by a general result on arithmetic subgroups of linear algebraic groups \cite[Theorem 6.12]{BH62} (See also Theorem \ref{thm11}). So, $S$ itself is not a candidate surface in Theorem \ref{thm1}.
\item $S$ is the unique minimal model of $Y$ up to isomorphisms. So, we have a birational morphism $\nu : Y \to S$, which is a finite composition of blowings up at points. Therefore, we have $H^0(Y, T_Y) = 0$ and also an injective group homomorphism
$${\rm Aut}\, (Y) \subset {\rm Bir}\, (S) = {\rm Aut}\, (S)\,\, ; \,\, f \mapsto \nu \circ f \circ \nu^{-1}\,\, ,$$
via $\nu$. Note that a subgroup of a finitely generated group is not necessarily finitely generated (cf. Theorem \ref{thm11}).
\end{enumerate}
\end{remark}

We show Theorem \ref{thm1} by constructing $Y$ explicitly. Our construction is inspired by \cite{Le18} for 6-dimensional examples, \cite{DO19} and \cite{Og19} for exmaples birational to K3 surfaces and also \cite{Mu10} for his new construction of an Enriques surface with a numerically trivial involution, which is missed in an earlier paper \cite{MN84}. As usual in study of Enriques surfaces, our construction is more involved than \cite{DO19} and \cite{Og19} for K3 surfaces, whereas the basic strategy of the construction is essentially the same.

As in \cite{DO19} and \cite{Og19}, the following purely group theoretical theorem (see eg. \cite{Su82}) will be frequently used in this paper.

\begin{thm}\label{thm11}
Let $G$ be a group and $H \subset G$ a subgroup of $G$. Assume that $H$ is of finite index, i.e., $[G :H] < \infty$. Then, the group $H$
 is finitely generated if and only if $G$ is finitely generated.
\end{thm}

In this paper, for a variety $V$ we denote
the group of biregular automorphisms of $V$ and the group of birational automorphisms of $V$ by
$${\rm Aut}\, (V),\,\, \,\, {\rm Bir}\, (V)$$
respectively, and for closed subsets  $W_1$, $W_2$, $\ldots$, $W_n$ of $V$ the decomposition group and the inertia group by
% $${\rm Dec}\, (W) := {\rm Dec}\, (V, W) := \{ f \in {\rm Aut}\,(V)\,\, |\,\, f(W) = W\}\,.$$
%Similarly, we denote
$${\rm Dec}\, (W_1, \ldots, W_n) := {\rm Dec}\, (V, W_1, \ldots, W_n) := \{ f \in {\rm Aut}\,(V)\,\, |\,\, f(W_i) = W_i (\forall i)\}\,\, ,$$
$${\rm Ine}\, (W_1, \ldots, W_n) := {\rm Ine}\, (V, W_1, \ldots, W_n) := \{ f \in {\rm Dec}\, (V, W_1, \ldots, W_n)\, |\, f_{W_i} = \id_{W_i} (\forall i)\}.$$
%for closed subsets $W_1$, $W_2$, $\ldots$, $W_n$ of $V$. Here ${\rm Dec}$ is an abrreviation of the decomposition group.
For basic properties of surfaces, we refer to \cite{BHPV04} and \cite{CD89}.

We believe that large part of our construction should work also in positive characteristic $\ge 3$ if the based field is carefully chosen (see e.g. for some sensitive aspect of the base field in positive characteristic \cite{Og19}). We leave it to the readers who are interested in this generalization.

\medskip\noindent
{\bf Acknowledgements.} We would like to thank Professor Jun-Muk Hwang for organizing one day workshop at KIAS,
% and his warm encouragement,
which made our collaboration possible, and Professor Yuya Matsumoto for very kind help concerning Figure \ref{fig1}.

%%%%%%%%
\section{Preliminaries}\label{sect2}

In this section, first we fix some basic notation concerning a Kummer surface ${\rm Km}\, (E \times F)$ of the product of two non-isogenous elliptic curves. Our notation follows \cite{DO19} and \cite{Og19}. Then we recall Mukai's construction of Enriques surfaces with a numerically trivial involution of odd type \cite{Mu10} arising from ${\rm Km}\, (E \times F)$. His construction is very crucial in our construction in Section \ref{sect3}.

\subsection{Kummer surfaces of product type}\label{sub1}

Let $E$ be the elliptic curve defined by the Weierstrass equation
$$y^2 = x(x-1)(x-t)\,\, ,$$
and $F$ be the elliptic curve defined by the Weierstrass equation
$$v^2 = u(u-1)(u-s)\,\, .$$
Note that $E/\langle -1_E \rangle = \P^1$, the associated quotient map $E \to \P^1$ is given by $(x,y)\mapsto x$ and the points $0$, $1$, $t$ and $\infty$ of $\P^1$ are exactly the branch points of this quotient map. The same holds for $F$ if we replace $t$ by $s$.

Throughout this paper, we make the following assumption:

\begin{assumption}\label{ass21} $t$ and $s$ are transcendental over $\Q$ and the two elliptic curves
$E$ and $F$ are not isogenous.
\end{assumption}
Assumption \ref{ass21} is satisfied if $s \in \C$ is generic with respect to a transcendental number $t \in \C$.

Let
$$X := {\rm Km} (E \times F)$$
be the Kummer K3 surface accociated to the product abelian surface $E \times F$, that is, the minimal resolution of the quotient surface $E \times F/\langle -1_{E\times F} \rangle$. We write $H^0(X, \Omega_X^2) = \C \omega_X$. Since $E$ and $F$ are not isogenous, the Picard number $\rho(X)$ of $X$ is $18$ (See eg.\cite[Prop. 1 and Appendix]{Sh75}).

Let $\{a_i\}_{i=1}^{4}$ and $\{b_i\}_{i=1}^{4}$ be the $2$-torsion subgroups of $F$ and $E$ respectively. Then $X$ contains 24 smooth rational curves which form the so called double Kummer pencil on $X$, as in Figure \ref{fig1}. Here smooth rational curves $E_i$, $F_i$ ($1 \le i \le 4$) are arising from the elliptic curves $E \times \{a_i\}$, $\{b_i\} \times F$ on $E \times F$. Smooth rational curves $C_{ij}$ ($1\le i,j \le 4$) are the exceptional curves over the $A_1$-singular points of the quotient surface $E \times F/\langle -1_{E\times F} \rangle$. Throughout this paper, we will freely use the names of curves in Figure \ref{fig1}.

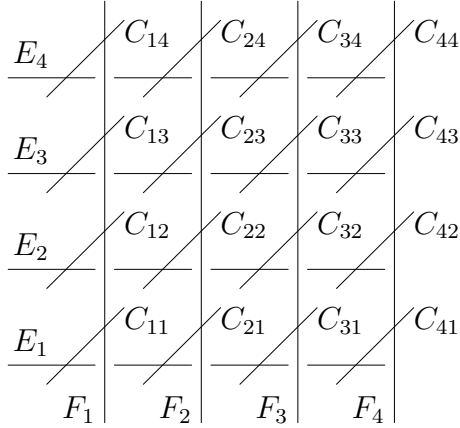
\begin{figure}
 %\centering
% \input{alldiv.tpc}
\unitlength 0.1in
\begin{picture}(25.000000,24.000000)(-1.000000,-23.500000)
\put(4.500000, -22.000000){\makebox(0,0)[rb]{$F_1$}}%
\put(9.500000, -22.000000){\makebox(0,0)[rb]{$F_2$}}%
\put(14.500000, -22.000000){\makebox(0,0)[rb]{$F_3$}}%
\put(19.500000, -22.000000){\makebox(0,0)[rb]{$F_4$}}%
\put(0.250000, -18.500000){\makebox(0,0)[lb]{$E_1$}}%
\put(0.250000, -13.500000){\makebox(0,0)[lb]{$E_2$}}%
\put(0.250000, -8.500000){\makebox(0,0)[lb]{$E_3$}}%
\put(0.250000, -3.500000){\makebox(0,0)[lb]{$E_4$}}%
\put(6.000000, -16.000000){\makebox(0,0)[lt]{$C_{11}$}}%
\put(6.000000, -11.000000){\makebox(0,0)[lt]{$C_{12}$}}%
\put(6.000000, -6.000000){\makebox(0,0)[lt]{$C_{13}$}}%
\put(6.000000, -1.000000){\makebox(0,0)[lt]{$C_{14}$}}%
\put(11.000000, -16.000000){\makebox(0,0)[lt]{$C_{21}$}}%
\put(11.000000, -11.000000){\makebox(0,0)[lt]{$C_{22}$}}%
\put(11.000000, -6.000000){\makebox(0,0)[lt]{$C_{23}$}}%
\put(11.000000, -1.000000){\makebox(0,0)[lt]{$C_{24}$}}%
\put(16.000000, -16.000000){\makebox(0,0)[lt]{$C_{31}$}}%
\put(16.000000, -11.000000){\makebox(0,0)[lt]{$C_{32}$}}%
\put(16.000000, -6.000000){\makebox(0,0)[lt]{$C_{33}$}}%
\put(16.000000, -1.000000){\makebox(0,0)[lt]{$C_{34}$}}%
\put(21.000000, -16.000000){\makebox(0,0)[lt]{$C_{41}$}}%
\put(21.000000, -11.000000){\makebox(0,0)[lt]{$C_{42}$}}%
\put(21.000000, -6.000000){\makebox(0,0)[lt]{$C_{43}$}}%
\put(21.000000, -1.000000){\makebox(0,0)[lt]{$C_{44}$}}%
\special{pa 500 2200}%
\special{pa 500 0}%
\special{fp}%
\special{pa 1000 2200}%
\special{pa 1000 0}%
\special{fp}%
\special{pa 1500 2200}%
\special{pa 1500 0}%
\special{fp}%
\special{pa 2000 2200}%
\special{pa 2000 0}%
\special{fp}%
\special{pa 0 1900}%
\special{pa 450 1900}%
\special{fp}%
\special{pa 550 1900}%
\special{pa 950 1900}%
\special{fp}%
\special{pa 1050 1900}%
\special{pa 1450 1900}%
\special{fp}%
\special{pa 1550 1900}%
\special{pa 1950 1900}%
\special{fp}%
\special{pa 0 1400}%
\special{pa 450 1400}%
\special{fp}%
\special{pa 550 1400}%
\special{pa 950 1400}%
\special{fp}%
\special{pa 1050 1400}%
\special{pa 1450 1400}%
\special{fp}%
\special{pa 1550 1400}%
\special{pa 1950 1400}%
\special{fp}%
\special{pa 0 900}%
\special{pa 450 900}%
\special{fp}%
\special{pa 550 900}%
\special{pa 950 900}%
\special{fp}%
\special{pa 1050 900}%
\special{pa 1450 900}%
\special{fp}%
\special{pa 1550 900}%
\special{pa 1950 900}%
\special{fp}%
\special{pa 0 400}%
\special{pa 450 400}%
\special{fp}%
\special{pa 550 400}%
\special{pa 950 400}%
\special{fp}%
\special{pa 1050 400}%
\special{pa 1450 400}%
\special{fp}%
\special{pa 1550 400}%
\special{pa 1950 400}%
\special{fp}%
\special{pa 200 2000}%
\special{pa 600 1600}%
\special{fp}%
\special{pa 200 1500}%
\special{pa 600 1100}%
\special{fp}%
\special{pa 200 1000}%
\special{pa 600 600}%
\special{fp}%
\special{pa 200 500}%
\special{pa 600 100}%
\special{fp}%
\special{pa 700 2000}%
\special{pa 1100 1600}%
\special{fp}%
\special{pa 700 1500}%
\special{pa 1100 1100}%
\special{fp}%
\special{pa 700 1000}%
\special{pa 1100 600}%
\special{fp}%
\special{pa 700 500}%
\special{pa 1100 100}%
\special{fp}%
\special{pa 1200 2000}%
\special{pa 1600 1600}%
\special{fp}%
\special{pa 1200 1500}%
\special{pa 1600 1100}%
\special{fp}%
\special{pa 1200 1000}%
\special{pa 1600 600}%
\special{fp}%
\special{pa 1200 500}%
\special{pa 1600 100}%
\special{fp}%
\special{pa 1700 2000}%
\special{pa 2100 1600}%
\special{fp}%
\special{pa 1700 1500}%
\special{pa 2100 1100}%
\special{fp}%
\special{pa 1700 1000}%
\special{pa 2100 600}%
\special{fp}%
\special{pa 1700 500}%
\special{pa 2100 100}%
\special{fp}%
\end{picture}%
 \caption{Curves $E_i$, $F_j$ and $C_{ij}$}
 \label{fig1}
\end{figure}

We denote the unique point $E_j \cap C_{ij}$ by $P_{ij}$ and the unique point $F_i \cap C_{ij}$ by $P_{ij}'$. We may and do adapt $x$ (resp. $u$) the affine coordinate of $E_j$ and $F_i$ so that
$$P_{1j} = 1\,\, ,\,\, P_{2j} = t\,\, ,\,\, P_{3j} = \infty\,\, ,\,\, P_{4j} = 0$$
on $E_j$ with respect to the coordinate $x$ and
$$P_{i1}' = 1\,\, ,\,\, P_{i2}' = s\,\, ,\,\, P_{i3}' = \infty\,\, ,\,\,
P_{i4}' = 0$$
on $F_i$ with respect to the coordinate $u$.

Set
$$\theta := [(1_E, -1_F)] = [(-1_E, 1_F)] \in {\rm Aut}\, (X)\,\, .$$
Then $\theta$ is an involution of $X$, i.e., an automorphism of $X$ of order $2$. The following lemma was proved in \cite[Lemmas (1.3), (1.4)]{Og89} (See also \cite{Og19}).

\begin{lemma}\label{lem21}
\begin{enumerate}
\item $\theta^* = \id$ on ${\rm Pic}\, (X)$ and $\theta^* \omega_X = -\omega_X$.\item $f \circ \theta = \theta \circ f$ for all $f \in {\rm Aut}\,(X)$.
\item Let $X^{\theta}$ be the fixed locus of $\theta$. Then $X^{\theta} = \cup_{i=1}^{4} (E_i \cup F_i)$.
\item ${\rm Aut}\,(X) = {\rm Dec}\, (X, \cup_{i=1}^{4} (E_i \cup F_i))$.
\end{enumerate}
\end{lemma}

\subsection{Enriques surfaces with a numerically trivial involution of odd type}\label{sub2}

We employ the same notation as in Subsection \ref{sub1}. By Assumption \ref{ass21}, the two ordered sets
$$\{P_{i1}', P_{i2}', P_{i3}', P_{i4}'\}\subset F_i\cong \P^1\,\, ,\,\, \{P_{1j}, P_{2j}, P_{3j}, P_{4j}\} \subset E_j\cong \P^1$$
are not projectively equivalent, i.e., not in the same orbit of the action of ${\rm Aut}\, (\P^1) = {\rm PGL}\, (2, \C)$ on $\P^1$.

%In particular,
We recall the construction of %Professor
 Mukai \cite{Mu10} for our $X = {\rm Km}\, (E \times F)$.
 %The rest follows \cite{Mu10}.
Let $$T := X/\langle \theta \rangle$$ be the quotient surface and $$q : X \to T$$ be the quotient morphism. Then $T$ is a smooth projective surface such that $q(C_{ij})$ ($1 \le i, j \le 4$) is a $(-1)$-curve, i.e., a smooth rational curve with self interesection number $-1$. Then $T$ is obtained by the blowings up of $\P^1 \times \P^1$ at the 16 points $p_{ij}$ ($1 \le i, j \le 4$) of $\P^1 \times \P^1$. We may assume that $p_{ij}$ is the image of $C_{ij}$ under the composite morphism
$$X \to T \to \P^1 \times \P^1\,\, .$$
Let us consider the Segre embedding
$$\P^1 \times \P^1 \subset \P^3\,\, ,$$
and identify $\P^1 \times \P^1$ with a smooth quadric surface
$Q$ in $\P^3$. Since the four points $p_{11}, p_{22}, p_{33}, p_{44} \in Q$  are not coplaner in $\P^3$, we may adjust coordinates $[x_1 : x_2 : x_3 : x_4]$ of $\P^3$ so that the $4$ points are
$$p_{11}=[1:0:0:0],\,\, \, p_{22}= [0:1:0:0],\,\, \, p_{33}=[0:0:1:0],\,\, \, p_{44}=[0:0:0:1]\,\, .$$
Then the equation of $Q$ is of the form
$$\alpha_1x_2x_3 + \alpha_2x_1x_3 + \alpha_3x_1x_2 + (x_1+x_2+x_3)x_4 = 0$$
for some complex numbers $\alpha_i$ satisfying non-degeneracy condition.
Then the Cremona involution of $\P^3$
$$\tilde{\tau}' : [x_1 :x_2 : x_3 : x_4] \mapsto [\frac{\alpha_1}{x_1} : \frac{\alpha_2}{x_2} : \frac{\alpha_3}{x_3} : \frac{\alpha_1\alpha_2\alpha_3}{x_4}]$$ satisfies
 $\tilde{\tau}'(Q) = Q$, hence induces a birational automorphism of $Q$
$$\tau' := \tilde{\tau}'|_{Q} \in {\rm Bir}\, (Q)\,\, .$$
Let $I(\tau')$ be the indeterminacy locus of $\tau'$.
By the definition of $\tau'$, we readily check the following (\cite[Section 2]{Mu10}):
\begin{lemma}\label{lem22}
\begin{enumerate}
\item $I(\tau') = \{p_{ii}\}_{i=1}^{4}$ and $\tau'$ contracts the (smooth) conic curve $C'_{i} := Q \cap (x_i= 0)$ to $p_{ii}$.
\item $\tau'$ interchanges the two lines through $p_{ii}$ for each $i=1$, $2$, $3$, $4$.
\item $\mu^{-1}\circ \tau' \circ \mu \in {\rm Aut}\, (B)$, where $\mu : B \to \P^1 \times \P^1$ is the blowing up at the four points $p_{ii}$ ($1 \le i \le 4$).
\end{enumerate}
\end{lemma}

By the property (2), $\tau'(p_{ij}) = p_{ji}$ if $1 \le i \not= j \le 4$. Therefore $\tau'$ lifts to $$\tau \in {\rm Aut}\, (T).$$ Since $q : X \to T$ is the finite double cover branched along the unique anti-bicanonical divisor $$\sum_{i=1}^{4} (q(E_i)+q(F_i))\in |-2K_T|,$$ it follows that $\tau$ lifts to an involution $$\epsilon \in {\rm Aut}\, (X).$$ Apriori, there are exactly the two choices of the lifting $\epsilon$; if we denote one lifting by $\epsilon_0$ then the other is $\epsilon_0 \circ \theta$. Recall that $\theta^*\omega_X = -\omega_X$. Thus, we may and do choose the unique lift $\epsilon$ with $\epsilon^* \omega_X = -\omega_X$.
Set
$$Z := X/\langle \epsilon \rangle\,\, .$$
and denote the quotient morphism by
$$\pi : X \to Z\,\, .$$
The following discovery due to %Professor Shigeru
Mukai \cite[Proposition 2]{Mu10} is also crucial for us:
\begin{proposition}\label{prop21}
The involution $\epsilon$ acts on $X$ freely. In particular,
$Z$ is an Enriques surface with a numerically trivial involution $\theta_Z \in {\rm Aut}\, (Z)$ induced from $\theta \in {\rm Aut}\, (X)$.
%$\iota$.
\end{proposition}
Set $$C_i := \epsilon (C_{ii})\,\, (i = 1,\,2,\,3,\, 4).$$
Then, $C_i$ is the proper transform of the curve $C_i'$ in Lemma \ref{lem22} under the morphism
$$X \to T \to B \to \P^1 \times \P^1 = Q\,\, .$$
\begin{corollary}\label{cor21} %Let $i$, $j$, $k$ be integers such that $1 \le i, j, k \le 4$. The involution $\epsilon \in {\rm Aut}\, (X)$ satisfies
\begin{enumerate}
\item $\epsilon (E_i) = F_i$, $\epsilon(F_i) = E_{i}$ for all $i =1$, $2$, $3$, $4$.
\item $\epsilon(C_{ij}) = C_{ji}$ for all $i$, $j$ such that $i \not= j$.
\item $(C_i, E_i) = (C_i, F_i) = 1$, $(C_i, C_{ii}) = 0$, $(C_i, C_{kj}) = 0$
for all $i$, $j$, $k$ such that $k \not= j$.
\item $(C_i, E_j) = (C_i, F_j) = 0$ for all $i$, $j$ such that $j \not= i$.
%\item The intersection numbers in (3) and (4) uniquely determine the
%class of $C_i$ in ${\rm Pic}\, (X) \simeq {\rm NS}\, (X)$.
\end{enumerate}
\end{corollary}
\begin{proof} The assertions (1) and (2) follow from the description of $\tau$. Then the assertions (3) and (4) follow from $\epsilon(C_{ii}) = C_{i}$ and the assertions (1) and (2), except possibly $(C_i, C_{ii}) = 0$. The latter follows from the fact that the conic curve $C'_{i} \subset Q$ that is contracted to $p_{ii}$ by $\tau'$ does not pass through $p_{ii}$ (See Lemma \ref{lem22} (1)).

%It remains to prove the assertion (5). Let $D$ be any divisor on $X$. Consider the elliptic fibration $p_2 : X \to \P^1$ with sections $F_i$, induced by the second projection ${\rm pr}_2 : E \times F \to F$. This fibration has four $I_0^*$-fibres. % ($I_0^*$ is $\Tilde{D_4}$ in $ADE$-notation).
%Since $\rho(X) = 18$, the Mordell-Weil group ${\rm MW}\, (p_2)$ is of rank $0$ \cite{Og89}. Hence ${\rm NS}\, (X) \otimes \Q$ is generated by the section class $[F_1]$ %, the fiber class $[F]$
%and the irreducible components of the four reducible singular fibers.
%Then it readily follows that ${\rm NS}\, (X) \otimes \Q$ is generated by the classes $[F_j]$, $[E_j]$, $[C_{jk}]$ ($j \not= k$) and one more class $[C_{ii}]$ (for any choice of $i$). Thus, the intersection numbers $(D, *)$ with these classes uniquely determine the class of $D$ in ${\rm NS}\, (X) \otimes \Q$, and therefore determine the class of $D$ in ${\rm Pic}\, (X) \simeq {\rm NS}\, (X)$ (as the Picard group of a K3 surface is torsion free). This completes the proof.
\end{proof}

\section{Construction and proof of Theorem \ref{thm1}}\label{sect3}

We employ the same notation and the assumption (Assumption \ref{ass21}) as in Section \ref{sect2}. For instance,
$$X = {\rm Km}\, (E \times F)\,\, ,\,\, Z = X/\langle \epsilon \rangle\,\, ,\,\, \pi : X \to Z\,\, .$$
We also use the following notation for curves and points on the Enriques surface $Z$:
$$H_j := \pi(E_j)\,\, ,
\,\, D_{ij} := \pi(C_{ij})\,\, , \,\, Q_{ij} := \pi(P_{ij})\,\, ,$$
and via the isomorphism $\pi|_{E_j} : E_j \to H_j$, we also regard $x$ as the affine coordinate of $H_j$. Then $Q_{ij} \in H_j$ and
$$x(Q_{1j}) = 1\,\, ,\,\, x(Q_{2j}) = t\,\, ,\,\, x(Q_{3j}) = \infty\,\, ,\,\, x(Q_{4j}) = 0\,\, .$$
By Corollary \ref{cor21}, we have
$$\pi^{-1}(H_j) = E_j \cup F_j$$
for each $j =1$, $2$, $3$, $4$ and
$$\pi^{-1}(D_{ij}) = C_{ij} \cup C_{ji}\,\, ,\,\, \pi^{-1}(Q_{ij}) = \{P_{ij},
P_{ji}'\}$$
if $i \not= j$, while
$$\pi^{-1}(D_{ii}) = C_{ii} \cup C_i\,\, ,\,\, \pi^{-1}(Q_{ii}) = \{P_{ii} \cup P_{i}\}\,\, ,$$
again for each $i$.
Here $P_{i}$ is the unique intersection point of $C_i \cap F_i$.

Let $\mu_1 : Z_1 \to Z$ be the blowing up at the point $Q_{32} \in H_2$, i.e., the blowing up at $\infty$ under the coordinate $x$ of $H_2$. Let
$$E_{\infty} := \P(T_{Z,Q_{32}}) \simeq \P^1$$
be the exceptional divisor of $\mu_1$. We then choose three mutually different points on $\P(T_{Z,Q_{32}})$, say $Q_{32k}$ ($k=1$, $2$, $3$). Let $\mu_2 : Z_2 \to Z_1$ be the blowings up of $Z_1$ at the three points $Q_{32k}$.

Our main theorem is Theorem \ref{thm31} below. Clearly, Theorem \ref{thm1} follows from Theorem \ref{thm31} by taking $Y = Z_2$:
\begin{theorem}\label{thm31}
${\rm Aut}\, (Z_2)$ is not finitely generated.
\end{theorem}

In the rest of this section, we prove Theorem \ref{thm31}.

We denote
$$\mu := \mu_1 \circ \mu_2 : Z_2 \to Z_1 \to Z\,\, .$$
By $E_{32k}$, we denote the exceptional curve over $Q_{32k}$ under $\mu_2$ and
by $E_{\infty}'$ the proper transform of $E_{\infty}$ under $\mu_2$.

First we reduce the proof to $Z$. For this, we recall that
$${\rm Aut}\, (Z_2) \subset {\rm Aut}\, (Z)$$
via $\mu$ (See Remark \ref{rem1}).
We define
$${\rm Ine}\,(Z, Q_{32}, T_{Q_{32}}) := \{ f \in {\rm Dec}\,(Z, Q_{32})\,|\, df|_{T_{Z, Q_{32}}} = \id_{T_{Z, Q_{32}}}\}\,\, .$$

\begin{proposition}\label{prop31}
\begin{enumerate}
\item There is a subgroup $K$ of
${\rm Aut}\, (Z_2)$ such that $[{\rm Aut}\, (Z_2) : K] < \infty$, ${\rm Ine}\,(Z, Q_{32}, T_{Q_{32}}) \subset K$ via $\mu$ and $[K : {\rm Ine}\,(Z, Q_{32}, T_{Q_{32}})] < \infty$.

\item If ${\rm Ine}\,(Z, Q_{32}, T_{Q_{32}})$ is not finitely generated, then ${\rm Aut}\, (Z_2)$ is not finitely generated.
\end{enumerate}
\end{proposition}

\begin{proof} First we show (1). By the canonical bundle formula, we have
$$|2K_{Z_2}| = \{2E_{\infty}' + 4(E_{321} + E_{322} + E_{323})\}\,\, .$$
Since ${\rm Aut}\, (Z_2)$ preserves $|2K_{Z_2}|$, it follows that
$${\rm Aut}\, (Z_2) = {\rm Dec}\,(Z_2, E_{\infty}', E_{321} \cup E_{322} \cup E_{323})\,\, .$$
Therefore, via $\tau_2$, we have
$${\rm Aut}\, (Z_2) = {\rm Dec}\,(Z_1, E_{\infty}, \{Q_{321}, Q_{322}, Q_{323}\}) \subset {\rm Aut}\, (Z_1)\,\, .$$
Thus, the group
$$K :=
{\rm Dec}\,(Z_1, E_{\infty}, \{Q_{321}\}, \{Q_{322}\}, \{Q_{323}\})$$
is a subgroup of ${\rm Aut}\, (Z_2)$ with $[{\rm Aut}\, (Z_2) : K] \le 6 = |{\rm Aut}_{{\rm set}}\,(\{Q_{321}, Q_{322}, Q_{323}\})|$.

We will show that $K$ satisfies the requirement.

Since only $\id_{\P^1}$ is the automorphism of $\P^1$ pointwisely fixes three points, it follows that
$$K = {\rm Ine}\,(Z_1, E_{\infty})\,\, .$$
Since $E_{\infty} = \P(T_{Z, Q_{32}})$, we deduce that
$$K = \{f \in {\rm Dec}\, (Z, Q_{32})\,|\, df|_{T_{Z, Q_{32}}} = \alpha(f)
\id_{T_{Z, Q_{32}}}\,\, (\exists\alpha(f) \in \C^{\times})\} \subset {\rm Dec}\, (Z, Q_{32})\,\, .$$
Observe that if $df|_{T_{Z, Q_{32}}} = \alpha(f) \id_{T_{Z, Q_{32}}}$ for $f \in K$, then
$$(df \wedge df)^{\otimes 2}|_{(\wedge^2 T_{Z, Q_{32}})^{\otimes 2}} = \alpha(f)^4 \id_{(\wedge^2 T_{Z, Q_{32}})^{\otimes 2}}\,\, .$$
Since the line bundle $(\Omega_Z^2)^{\otimes 2}$ admits a nowhere vanishing global section, it follows that $\alpha(f)^4$ is in the image ${\rm Im}\, r_2$ of the bicanonical representation
$$r_2 : {\rm Aut}\, (Z) \to {\rm GL}(H^0(Z, (\Omega_Z^2)^{\otimes 2})) \simeq \C^{\times}$$
of ${\rm Aut}\, (Z)$ (\cite[Section 14]{Ue75}). Since ${\rm Im}\, r_2$
is finite by \cite[Theorem 14.10]{Ue75}, it follows that
$\{\alpha(f)\, |\, f \in K\}$ is also finite.
Hence $[K : {\rm Ine}\,(Z, Q_{32}, T_{Q_{32}})] < \infty$ as well.

Let us show (2). Recall Theorem \ref{thm11}. Then, if ${\rm Ine}\,(Z, Q_{32}, T_{Q_{32}})$ is not finitely generated, then $K$
is not finitely generated by $[K : {\rm Ine}\,(Z, Q_{32}, T_{Q_{32}})] < \infty$. Hence ${\rm Aut}\, (Z_2)$ is not finitely generated, again by $[{\rm Aut}\, (Z_2) : K] < \infty$.
\end{proof}

In what follows, we will show that ${\rm Ine}\,(Z, Q_{32}, T_{Q_{32}})$ is not finitely generated. This is a problem on the Enriques surface $Z$.

\begin{lemma}\label{lem31}
\begin{enumerate}
\item Let $f \in {\rm Dec}\, (Z, Q_{32})$. Then $f(H_2) = H_2$, i.e., $f \in {\rm Dec}\, (Z, H_2)$.

\item The differential maps $df|_{T_{Z, Q_{32}}}$ for all $f \in {\rm Dec}\, (Z, Q_{32})$ are simultaneously diagonalizable.

\item Let $f \in {\rm Ine}\,(Z, Q_{32}, T_{Q_{32}})$. Then $f \in {\rm Dec}\, (Z, H_2)$ by (1)
and
$$d(f|_{H_2})|_{T_{H_2, Q_{32}}} = \id_{T_{H_2, Q_{32}}}$$
for the induced action.
\end{enumerate}
\end{lemma}
\begin{proof}
Let $f \in {\rm Dec}\, (Z, Q_{32})$. Then, the one of the two lifts of $f$, say $\tilde{f}$, satisfies $\tilde{f}(P_{32}) = P_{32}$. Therefore the result follows from the corresponding result on $X$ (see eg. \cite{DO19}).

For the convenience of the readers, we recall the proof here from \cite{DO19}. Since $\tilde{f} \in {\rm Dec}(X, \cup_{j=1}^{4} (E_j \cup F_j))$ by Lemma \ref{lem21} (4) and $E_{2}$ is the unique component of $\cup_{j=1}^{4} (E_j \cup F_j)$, containinig $P_{32}$, it follows that $\tilde{f} \in {\rm Dec}\,(X, E_2)$. This shows (1).

By Lemma \ref{lem21} (1), (3), one has $\theta(R) = R$ for any smooth rational curve $R$ on $X$
and
$$d(\theta|_{E_2})_{P_{32}} = 1\,\, ,\,\, d(\theta|_{C_{32}})_{P_{32}} = -1\,\, .$$
In particular,
$$T_{X, P_{32}} = T_{E_2, P_{32}} \oplus T_{C_{32}, P_{32}}\,\, .$$
Note that $\tilde{f}(E_2) = E_2$ as observed above. Let $C_{32}' := \tilde{f}(C_{32})$. Then $P_{32} \in C_{32}' \simeq \P^1$ and the induced action $\theta|_{C_{32}'}$ satisfies
$$d(\theta|_{C_{32}'})_{P_{32}} = -1$$
by Lemma \ref{lem21} (1). Thus, $d\tilde{f}|_{T_{X, P_{32}}}$ for all $\tilde{f}$ preserve both $T_{E_2, P_{32}}$ and $T_{C_{32}, P_{32}}$. This implies (2).

The assertion (3) is now obvious.
\end{proof}
Recall that for $Q \in \P^1$,
$${\rm Ine}\, (\P^1, Q, T_{\P^1, Q}) := \{f \in {\rm Ine}(\P^1, Q)\, |\, df|_{T_{\P^1, Q}} = \id_{T_{\P^1, Q}}\} \simeq (\C, +)\,\, .$$
Here $(\C, +)$ is the additive group, in particular, an abelian group. The last isomorphism is given by
$$\C \ni c \mapsto (z \mapsto z + c) \in {\rm Ine}\, (\P^1, Q, T_{\P^1, Q})\,\, ,$$
if we choose an affine coordinate $z$ of $\P^1$ such that $z(Q) = \infty$.
By Lemma \ref{lem31} (3), we have then a representation
$$\rho : {\rm Ine}\,(Z, Q_{32}, T_{Q_{32}}) \to {\rm Ine}\, (H_2, Q_{32}, T_{H_{2}, Q_{32}}) \simeq (\C, +)\,\, .$$
Here, for the last isomorphism, we can use the affine coordinate $x$ of $H_2$ fixed at the beginning of this section.

\begin{proposition}\label{prop32}
\begin{enumerate}
\item There is $a \in \C \setminus \{0\}$ such that $t^{-2n}a \in {\rm Im}\,\rho$ for all positive integers $n$.
\item ${\rm Ine}\,(Z, Q_{32}, T_{Q_{32}})$
is not finitely generated.
\end{enumerate}
\end{proposition}
\begin{proof}
The assertion (2) follows from the assertion (1).
Indeed, the additive subgroup $M$ generated by $\{t^{-2n}a\, |\, n \in \Z_{\ge 0}\}$ is not finitely generated as $a \not= 0$ and $t$ is transcendental over $\Q$ by our assumption (Assumption \ref{ass21}). The assertion (1) says that $M \subset {\rm Im}\,\rho$. Since ${\rm Im}\, \rho \subset (\C, +)$, the group ${\rm Im}\, \rho$ is also an abelian group. It follows that the abelian group ${\rm Im}\,\rho$ is not finitely generated, either,
regardless of $[{\rm Im}\, \rho : M]$.
Hence ${\rm Ine}\,(Z, Q_{32}, T_{Q_{32}})$ is not finitely generated as claimed.

In the rest, we will show the assertion (1) by constructing two genus one fibrations on $Z$ and by considering their Jacobian fibrations.

Consider the following two divisors $M_1$ and $M_2$ of Kodaira's type $I_8$ and $IV^*$ on $Z$:
$$M_1 := H_2 + D_{32} + H_3 + D_{31} + H_1 + D_{41} + H_4 + D_{42}\,\, ,$$
$$M_2 := H_2 + 2D_{32} + H_1 + 2D_{31} + H_4 + 2D_{34} + 3H_3\,\, .$$
Then $|M_1|$ and $|M_2|$ define genus one fibrations
$$\varphi_{M_1} : Z \to \P^1\,\, ,\,\, \varphi_{M_2} : Z \to \P^1\,\, .$$
$\varphi_{M_1}$ is the genus one fibration induced from an elliptic fibration
$\Phi_1 : X \to \P^1$ on $X$ given by the divisor of Kodaira's type $I_8$
$$N_1 :=  E_2 + C_{32} + F_3 + C_{31} + E_1 + C_{41} + F_4 + C_{42}\,\, ,$$
and $\varphi_{M_2}$ is the genus one fibration induced from an elliptic fibration $\Phi_2 : X \to \P^1$ on $X$ given by the divisor of Kodaira's type $IV^*$
$$N_2 :=  E_2 + 2C_{32} + E_1 + 2C_{31} + E_4 + 2C_{34} + 3F_3\,\, .$$
By the classification of \cite[Theorem 2.1]{Og89}, $\Phi_1$ then belongs to Type ${\mathcal J}_1$ and $\Phi_2$ belongs to Type ${\mathcal J}_3$ in \cite[Theorem 2.1]{Og89}. By the definition of the action of our Enriques involution $\epsilon$ on $X$ and the classification of \cite[Theorem 2.1]{Og89}, it follows that the reducible fibers of $\Phi_1$ are exactly $N_1$ and $\epsilon(N_1)$, and the reducible fibers of $\Phi_2$ are exactly $N_2$ and $\epsilon(N_2)$. Thus, $\varphi_{M_1}$ has no reducible fibers other than $M_1$ and $\varphi_{M_2}$ has also no reducible fibers other than $M_2$.

Let us consider the (proper non-singular, relatively minimal) Jacobian fibration $\varphi_i : R_i \to \P^1$ of $\varphi_{M_i}$ for $i=1$ and $2$. Then the fiber $R_{i, p}$ of $\varphi_i$ over general $p \in \P^1$ is ${\rm Pic}^0\,(Z_{i, p})$, i.e., the identity component of the Picard group of the corresponding fiber $Z_{i, t}$ of $\varphi_{M_i}$. Therefore, the Mordell-Weil group ${\rm MW}\,(\varphi_i)$ of $\varphi_i$ acts on $\varphi_{M_i}$, which is the unique biregular extension of the translation action of ${\rm Pic}^0\,(Z_{i, p})$ on $Z_{i, p}$ where $p \in \P^1$ runs through general points. Note also that the types of singular fibers are the same for $\varphi_{M_i}$ and $\varphi_i$ up to multiplicities \cite[Theorem 5.3.1]{CD89}. Therefore $c_2(R_i) = c_2(Z) = 12$. In particular, $R_i \to \P^1$ are rational elliptic surfaces.

Here and hereafter, we will use basic notions and properties of Mordell-Weil lattices due to Shioda \cite{Sh90}.

Let us consider first the action of ${\rm MW}\,(\varphi_2)$ on $\varphi_{M_2} : Z \to \P^1$. From the fact that $\varphi_{M_2}$ has also no reducible fibers other than $M_2$, we see that $\varphi_2 : R_2 \to \P^1$ belongs to
No. 27 in the classfication of \cite[Main Theorem]{OS91}. Then, the narrow Mordell-Weil lattice ${\rm MW}^0\,(\varphi_2)$ of $\varphi_2$ is isomorphic to the positive definite root lattice $A_2$. In particular, there is $P \in {\rm MW}^0\,(\varphi_2)$ such that $\langle P, P\rangle =2$ for the height pairing of ${\rm MW}^0(\varphi_2)$ \cite[Section 8]{Sh90}. For this $P$, we have $(P) \cap (O) = \emptyset$ by \cite[Formula 8.19]{Sh90}.  Here $(P)$ is the divisor on $R_2$ corresponding to $P$. The action $t_P$ of $P$ on $\varphi_{M_2} : Z \to \P^1$ then preserves each irreducible component of $M_2$ as $P \in {\rm MW}^0\,(\varphi_2)$, particularly the curve $H_2$ and the point $Q_{32} \in H_2$, and the action $t_P|_{H_2}$ is of the form
$$x \mapsto x + a$$
for some $a \not= 0$ under the affine coordinate $x$ of $H_2$. Recall that the action of $d(t_P)$ on $T_{Z, Q_{32}}$ is diagonalizable (Lemma \ref{lem31}). Then, by the finiteness of bicanonical representation \cite[Theorem 14.10]{Ue75}, by replacing $t_P$ by some power $t_P^k$ ($k \not= 0$) and $a$ by $ka$ if necessary, we obtain an element
$$f_2 \in {\rm Ine}\,(Z, Q_{32}, T_{Q_{32}})$$
such that  $\rho(f_2) = a \not= 0$.

Next we consider the Jacobian fibration $\varphi_1 : R_1 \to \P^1$.
We need an explicit geometric construction of $\varphi_1$ from $\varphi_{M_1}$ explained by \cite[Lemma 2.6]{Ko86} and \cite[Section 3]{HS11}. Note that $D_{21}$ is a $2$-section of $\varphi_{M_1}$ and $\pi^{-1}(D_{21}) = C_{12} \cup C_{21}$. The curves $C_{12}$ and $C_{21}$ are sections of $\Phi_1$.
We may and do choose $C_{21}$ as the zero section of $\Phi_1$ and set
$$0 := [C_{21}] \in {\rm MW}(\Phi_1)\,\, .$$
Here and hereafter, we use the following notation:

\medskip\noindent
{\bf Notation.} 
\begin{enumerate}\item For a section $D$ of $\Phi_1$, we denote by $[D]$ the corresponding element of ${\rm MW}(\Phi_1)$ with respect to the zero section $C_{21}$. 
\item We denote by $T(R) \in {\rm Aut}\, (X)$ the automorphism corresponding to $R \in {\rm MW}\, (\Phi_1)$.
\end{enumerate}

\medskip\noindent
Then the element $[C_{12}] \in {\rm MW}(\Phi_1)$ is a $2$-torsion, because
$$\langle [C_{12}], [C_{12}] \rangle = 2\cdot2 + 2\cdot2 - \frac{4(8-4)}{8} - \frac{4(8-4)}{8} =0$$
for the height pairing \cite[Theorem 8.6, Formula 8.10]{Sh90} and ${\rm MW}(\Phi_1) \simeq \Z^{\oplus 2} \oplus \Z/2$ by \cite[Theorem 2.1, Case
${\mathcal J}_1$]{Og89}. 
Set
$$\iota := T([C_{12}]) \circ \epsilon \in {\rm Aut}\, (X)\,\, .$$
Then $\iota$ is an involution on $X$ (\cite[Lemma 2.6]{Ko86}) such that $X^{\iota}$ consists of two elliptic curves corresponding to the multiple fibers of $\varphi_{M_{2}}$ by Assumption \ref{ass21}. Then, by \cite[Lemma 2.6]{Ko86} (see also \cite[Section 3]{HS11}), the Jacobian fibration $\varphi_1$ of $\varphi_{M_1}$ is given by
$$\varphi_1 : R_1 = X/\langle \iota \rangle \to \P^1/\langle \epsilon \rangle
\,\, . $$
Here $\P^1/\langle \epsilon \rangle$ is the quotient of the base space $\P^1$ of $\Phi_1$ on which $\epsilon$ acts equivariantly as an involution. Let us denote by $\pi_{R_1} : X \to R_1$ the quotient morphism and the fibers $\pi_{R_1}(N_1)$ by $N_{1, R_1}$ and $\pi_{R_1}(X_p)$ by $R_{1, \overline{p}}$.

We may and do identify both $X_{p}$ and $X_{\epsilon(p)}$ with $R_{1, \overline{p}}$ for general $p \in \P^1$ via $\pi_{R_1}$.

Since $\iota$, $T([C_{12}])$ and $\epsilon$ are involutions, we have
$$\iota := T([C_{12}]) \circ \epsilon = \epsilon \circ T([C_{12}])\,\, .$$
Also by the construction, we find that
$$\iota(C_{21}) = \epsilon \circ T([C_{12}])(C_{21}) = \epsilon (C_{12}) = C_{21}\,\, ,$$
i.e., preservation of the zero section $C_{21}$ under $\iota$. Therefore, $Q \in {\rm MW}\,(\Phi_1)$ is induced from some element $Q'\in {\rm MW}(\varphi_1)$ exactly when
$$\iota \circ T([Q]) = T([Q]) \circ \iota\,\, {\rm ,i.e.,}\,\, \iota \circ T([Q]) \circ  \iota = T([Q])\,\, .$$
\begin{lemma}\label{lem32}
\begin{enumerate}
\item $\iota(C_{11}) = C_2$ and $\iota(C_2) = C_{11}$.
\item $[C_{11}]+[C_2]\in {\rm MW}\,(\Phi_1)$ is induced from some element $Q'\in {\rm MW}(\varphi_1)$.
\end{enumerate}
\end{lemma}

\begin{proof} By preservation of the zero section $C_{21}$ under $\iota$, we obtain that
$$\iota \circ T([C_{11}]+[C_2]) \circ \iota (x) = \iota (\iota(x) + [C_{11}] + [C_2]) = x + [\iota(C_{11})] + [\iota(C_2)] \,\, ,$$
for any $x \in X_p$ on each smooth fiber $X_p$. Hence
$$\iota \circ T(([C_{11}]+[C_2])) \circ \iota = T([\iota(C_{11})] + [\iota(C_2)])\,\, .$$
So, the assertion (2) follows from the assertion (1). We show the assertion (1).
Note that the torsion group of ${\rm MW}\, (\Phi_1)$ is isomorphic to $\Z/2$ by \cite[Theorem 2.1, Case ${\mathcal J}_1$]{Og89}. In particular, the non-zero torsion element is only $[C_{12}]$.

If we choose $C_{11}$ (instead of $C_{21}$) as the zero section of $\Phi_1$, then, the height pairing of the section $C_{22}$ with respect to the zero section $C_{11}$ is computed as
$$\langle C_{22}, C_{22}\rangle = 2\cdot2 + 2\cdot2 - \frac{4(8-4)}{8} - \frac{4(8-4)}{8} =0\,\, .$$
Thus $[C_{22}] - [C_{11}]$ is a non-zero torsion element in ${\rm MW}\, (\Phi_1)$ and therefore coincides with $[C_{12}]$, i.e.,
$$[C_{22}] = [C_{11}] + [C_{12}]$$
in ${\rm MW}\, (\Phi_1)$.
Since $\iota = \epsilon \circ T([C_{12}])$, it follows that
$$\iota (C_{11}) = \epsilon \circ T([C_{12}])(C_{11}) = \epsilon (C_{22}) = C_{2}\,\, $$
as claimed. Then
$$\iota(C_2) = \iota (\iota(C_{11})) = C_{11}\,\, ,$$
as $\iota$ is an involution. This completes the proof of Lemma \ref{lem32}.
\end{proof}
Let $Q' \in {\rm MW}\,({\rm MW}(\varphi_1)$ be as in Lemma \ref{lem32}. Then $Q'$ induces an automorphism $f_2 \in {\rm Aut}\, (Z)$ preserving each fiber of $\varphi_{M_1}$. The action of $f_2$ on $M_1 \setminus {\rm Sing}\, M_1 = \C^{\times} \times \Z/8$ \cite[Page 604]{Ko63} is then the same action of $Q'$ on $N_{1, R} \setminus {\rm Sing}\, (N_{1, R})$ and therefore also the same action of $[C_{11}] + [C_2]$ on $N_1 \setminus {\rm Sing}\, N_1$ under the identifications of these
three fibers by $\pi$ and $\pi_{R_1}$.
Thus, representing points on $M_1 \setminus {\rm Sing}\, M_1 = \C^{\times} \times \Z/8$
by $(x, m\,{\rm mod}\, 8)$, we have by \cite[Theorem 9.1, Page 604]{Ko63}
$$f_2 : (x,m\,{\rm mod}\, 8) \mapsto (tx, m\,{\rm mod}\, 8) \mapsto (tx, m + 4 \,{\rm mod}\, 8)\,\, .$$
Here we recall that $C_{22} \cap E_2 = t$ (resp. $C_2 \cap F_2 = t$) with respect to the affine coordinate $x$ on $E_2$ (resp. $u$ on $F_2$). Hence
$f_2^2(H_2) = H_2$, $f_2^2(Q_{32}) = Q_{32}$ and
$$f_2^2(x) = t^2x$$
on $H_2$. Then
$$(f_2^2)^{-n} \circ f_1 \circ (f_2^2)^{n} \in {\rm Ine}\,(Z, Q_{32}, T_{Z, Q_{32}})$$
and
$$(f_2^2)^{-n} \circ f_1 \circ (f_2^2)^{n}|_{H_2} : x \mapsto t^{2n}x \mapsto t^{2n}x +a \mapsto x + t^{-2n}a\,\,$$
on $H_2$. Thus
$$t^{-2n}a = \rho((f_2^2)^{-n} \circ f_1 \circ (f_2^2)^{n}) \in {\rm Im}\, \rho\,\, ,$$
as claimed.
 This completes the proof.
\end{proof}

Theorem \ref{thm31}, hence Theorem \ref{thm1}, now follows from Propositions \ref{prop31} (2) and \ref{prop32} (2).

%%%%%%%%%%%%%%%%%%%%%%%%%%%%%%%%%%%%%%%%%%%%%


\begin{thebibliography}{BHPV04}

\bibitem[BHPV04]{BHPV04} Barth, W., Hulek, K., Peters, C., Van de Ven, A., : \textit{Compact complex surfaces}. Second enlarged edition. Springer Verlag, Berlin-Heidelberg (2004).

\bibitem[BH62]{BH62} Borel, A., Harish-Chandra, : \textit{Arithmetic subgroups of algebraic groups}, Ann. of Math. {\bf 75} (1962) 485--535.

\bibitem[CD89]{CD89} Cossec, R.; Dolgachev, I. : \textit{Enriques surfaces. I}, Progress in Mathematics, {\bf 76}. Birkhauser Boston, Inc., Boston, MA, 1989.

\bibitem[DO19]{DO19} Dinh, T.-C., Oguiso, K., : \textit{A surface with discrete and nonfinitely generated automorphism group}, Duke Math. J. {\bf 168} (2019) 941--966.

\bibitem[Do84]{Do84} Dolgachev, I., : \textit{On automorphisms of Enriques surfaces}, Invent. Math. {\bf 76} (1984) 163--177.

\bibitem[HS11]{HS11} Hulek, K., Schuett, M., : \textit{Enriques surfaces and Jacobian elliptic K3 surfaces}, Math. Z. {\bf 268} (2011) 1025--1056.

\bibitem[Ko63]{Ko63} Kodaira, K., : \textit{On compact analytic surfaces. II}, Ann. of Math. {\bf 77} (1963), 563--626.

\bibitem[Ko86]{Ko86} Kondo, S., :\textit{ Enriques surfaces with finite automorphism groups}, Japan. J. Math. {\bf 12} (1986) 191--282.

\bibitem[Le18]{Le18} Lesieutre, J., :  \textit{A projective variety with discrete non-finitely generated automorphism group}, Invent. Math. {\bf 212} (2018) 189--211.

\bibitem[Mu10]{Mu10} Mukai, S., : \textit{Numerically trivial involutions of Kummer type of an Enriques surface}, Kyoto J. Math. {\bf 50} (2010) 889--902.

\bibitem[MN84]{MN84} Mukai, S., Namikawa, Y., : \textit{Automorphisms of Enriques surfaces which act trivially on the cohomology groups},
Invent. Math. {\bf 77} (1984) 383--397.

\bibitem[Og89]{Og89} Oguiso, K., : \textit{On Jacobian fibrations on the Kummer surfaces of the product of non-isogenous elliptic curves}, J. Math. Soc. Japan {\bf 41} (1989) 651--680.

\bibitem[Og19]{Og19} Oguiso, K., : \textit{A surface in odd characteristic with discrete and non-finitely generated automorphism group},
arXiv:1901.01351.

\bibitem[OS91]{OS91} Oguiso, K., Shioda, T., : \textit{The Mordell-Weil lattice of a rational elliptic surface}, Comment. Math. Univ. St. Paul.
{\bf 40} (1991) 83--99.

\bibitem[Sh75]{Sh75} Shioda, T., \textit{Algebraic cycles on certain
K3 surfaces in characteristic p}, Manifolds-Tokyo 1973 (Proc. Internat. Conf., Tokyo, 1973) 357--364. Univ. Tokyo Press, Tokyo 1975.

\bibitem[Sh90]{Sh90} Shioda, T., : \textit{On the Mordell-Weil
lattices}, Comment. Math. Univ. St. Paul. {\bf 39} (1990) 211--240.

\bibitem[Su82]{Su82} Suzuki, M., : \textit{Group theory. I}, Grundlehren der Mathematischen Wissenschaften, {\bf 247},  Springer-Verlag, Berlin-New York, 1982.
\bibitem[Ue75]{Ue75} Ueno, K., : \textit{Classification theory of algebraic varieties and compact complex spaces}. Lecture Notes in Mathematics, {\bf 439}, Springer-Verlag, 1975.


\end{thebibliography}
\end{document}